\documentclass[11pt]{amsart}
\usepackage{amscd,amsfonts,amsmath,amssymb,amstext,amsthm,graphicx,tikz,subcaption,mathdots}

\usetikzlibrary{arrows.meta}
\date{}
\def\Aut{\operatorname{Aut}}
\def\dim{\operatorname{dim}}
\def\cl{\operatorname{cl}}

\theoremstyle{plain}
\newtheorem{theorem} {Theorem} [section]
\newtheorem{lemma} [theorem]{Lemma}
\newtheorem{proposition}[theorem]{Proposition}

\theoremstyle{definition}

\newcommand{\dynkinstep}{1.0cm}

\newcommand{\dynkinnode}[3]{\node at (\dynkinstep*#1, \dynkinstep*#2) {#3};}
\newcommand{\dynkinarrow}[4]
{\draw[->] (\dynkinstep*#1,\dynkinstep*#2) -- (\dynkinstep*#3,\dynkinstep*#4);}

 \title[Rigidity of smooth Schubert varieties]{Rigidity of smooth Schubert varieties \\ in a  rational homogeneous manifold\\ associated to a short root}

\author[J. Hong]{Jaehyun Hong}

\address{Korea Institute for  Advanced  Study, 85 Hoegiro, Dongdaemun-gu, Seoul 02455, Korea}
\email{jhhong00@kias.re.kr}

\author[M. Kwon]{Minhyuk Kwon}
\address{Department of Mathematical Sciences \\
Seoul National University\\
599 Gwanak-ro Gwanak-gu  Seoul 151-742, Korea}
\email{mu10836@snu.ac.kr}

\begin{document}

\begin{abstract}
We classify smooth Schubert varieties $S_0$ in a rational homogeneous manifold $S$  associated to a short root, and  show that they are rigid 
 in the sense that  any subvariety of $S$ having the same homology class as $S_0$ is induced by the action of $\Aut_0(S)$, unless $S_0$ is linear.
\end{abstract}

 \maketitle

\section{Introduction}

Let $S=G/P$ be a rational homogeneous manifold associated to a simple root $\alpha_k$.  
The identity component $\Aut_0(S)$ of the automorphism group  of $S=G/P$
is equal to $G$ excepting the cases where
$(G, \{\alpha_k\})$ is $(B_{\ell}, \{\alpha_{\ell}\})$,
   $(C_{\ell}, \{\alpha_1\})$ or $(G_2, \{\alpha_1\})$.
In these cases, we will think of $S=G/P$ as a rational  homogeneous manifold $G'/P'$ with $\Aut_0(S)=G' \supsetneq G$.
The ample generator of the Picard group of $S$ induces a $G$-equivariant embedding of $S$ into a projective space.

 Under the action of a Borel subgroup of $G$, $S$  has only finitely many orbits.  
These orbits  give  rise to  a cell decomposition of $S$, so that the homology space of $S$ is generated freely by the  homology classes of their closures, Schubert varieties. In particular, the homology class of a (complex) subvariety of $S$ is a linear combination of the homology classes of Schubert varieties with nonnegative coefficients.

Homogeneous submanifolds   associated to   subdiagrams   of the marked Dynkin diagram  of $S$ are smooth Schubert varieties of $S$, and these are all smooth Schubert varieties when $S$ is associated to a long root (Proposition 3.7 of \cite{HoM}). They are rigid except for certain linear spaces $S_0$ in a rational homogeneous manifold $S$ associated to  a short root.

 \begin{theorem}  [Theorem 1.1 of \cite{HoM}]  \label{known result I} Let $S=G/P$ be a rational homogeneous manifold associated to a   simple root  and let $S_0=G_0/P_0$ be a homogeneous submanifold   associated to a subdiagram $\mathcal D(S_0)$ of the marked Dynkin diagram  $\mathcal D(S)$ of $S$.
  Then   any subvariety of $S$ having the same homology class as $S_0$ is induced by the action of $\Aut_0(S)$,
  excepting when $(S, S_0)$ is given by
\begin{enumerate}
\item [\rm (a)] $S=(C_{n}, \alpha_k)$, $S_0=\mathbb P^{b-k }$, $ \Lambda=\{ \alpha_{k-1}, \alpha_b \}$,  $ 2 \leq k < b \leq n ${\rm ;}
\item [\rm (b)] $S=(F_4, \alpha_3)$,  $S_0=\mathbb P^3$ or $\mathbb P^1$,   $\Lambda=\{ \alpha_1, \alpha_4\}$  or   $ \{ \alpha_2, \alpha_4\}${\rm ;}
\item [\rm (c)] $S=(F_4, \alpha_4)$,  $ S_0=\mathbb P^2 $ or $\mathbb P^1$,  $\Lambda =\{\alpha_2\}$ or    $ \{ \alpha_3 \}$
\end{enumerate}
\noindent
where $\Lambda$  denotes  the set of simple roots in $\mathcal D(S) \backslash \mathcal D(S_0)$ which are adjacent to the subdiagram $\mathcal D(S_0)$.

\end{theorem}

On the other hand there are non-homogeneous smooth Schubert varieties when $S$ is associated to a short root. For example, an odd symplectic Grassmannian in the symplectic Grassmannian $Gr_{\omega}(k,V)$, which was  introduced in (\cite{Mi}), is a smooth Schubert variety but is not homogeneous.
Here, $(V, \omega)$ is a complex vector space of dimension $2n$ with a symplectic form $\omega$ and $Gr_{\omega}(k,V)$ is the variety consisting of $\omega$-isotropic $k$-subspaces of $V$.
Fix an isotropic flag $F_{\bullet}:F_0  \subsetneq F_1 \subsetneq \dots \subsetneq F_{2n} =V$.  The subvariety  $Gr_{\omega}(k,V;F_a, F_{2n-1-a})$  of $Gr_{\omega}(k,V)$ consisting of $\omega$-isotropic subspaces of $V$, which contain $F_a$ and which are contained in $F_{2n-1-a}$, is called    an odd symplectic Grassmannian.  A smooth Schubert variety  of the symplectic Grassmannian is either a homogeneous submanifold associated to a subdaigram of the marked Dynkin diagram of $Gr_{\omega}(k,V)$, an odd symplectic Grassmannian, or a linear space (Theorem 1.2 of \cite{HoC}). Furthermore, an odd symplectic Grassmannian $S_0=Gr_{\omega}(k,V; F_a, F_{2n-1-a})$ for
  $0 \leq a \leq k-2$ is rigid, in the same sense as in Theorem \ref{known result I}, that is, any subvariety of $S=Gr_{\omega}(k,V)$ having the same homology class as $S_0$ is induced by the action of $\Aut_0(S)=\mathbb P Sp (V, \omega)$.  (Theorem 1.2 of \cite{HoM}). \\

  In this paper we will extend these results to other pair $(S,S_0)$ consisting of a rational homogeneous manifold $S$ associated to a short root and a   smooth Schubert variety $S_0$ of $S$. For the history and background of this kind of  rigidity problem, see \cite{HoM}.   Linear spaces of $S$ are classified in (\cite{LM}): a connected component of the space of linear spaces in $S$ corresponds to a linear Schubert variety of $S$. Some connected components have more than one $G$-orbits, i.e., for some linear Schubert varieties of  $S$ there is a deformation in $S$ which is not obtained by the action of $G$ (For details see \cite{LM}).
From now on, we will focus on non-linear smooth Schubert varieties.

\begin{theorem} \label{theorem main}
 Let $S=G/P$ be a rational homogeneous manifold of type $(F_4, \alpha_3)$ or of type $(F_4, \alpha_4)$. Then a non-linear smooth Schubert variety $S_0$ of $S$ is either a homogenous submanifold associated to a subdiagram of the Dynkin diagram of  $S$ or a horospherical variety embedded into $S$ of the following form:

 \begin{enumerate} 
 \item[\rm(1)] $S_0=(C_2, \alpha_2, \alpha_1)$ and $S=(F_4, \alpha_3)$;
 \item[\rm(2)] $S_0 =(B_3, \alpha_2, \alpha_3)$ and $S=(F_4, \alpha_3)$.
 \end{enumerate} 
 Furthermore,  
 any subvariety of $S$ having the same homology class as $S_0$ is induced by the action of $\Aut_0(S)$.
\end{theorem}

Together with Theorem 1.2 of \cite{HoC} and  Theorem 1.2 of \cite{HoM} for the case where $S$ is the symplectic Grassmannian $Gr_{\omega}(k,V)$, which are explained in the above, we get the following result.

\begin{theorem}    \label{Theorem 2} Let $S=G/P$ be a rational homogeneous manifold associated to a short  root. Then a non-linear smooth Schubert variety $S_0$ of $S$ is either a homogenous submanifold associated to a subdiagram of the Dynkin diagram of  $S$ or a horospherical variety embedded into $S$ of the following form:

 \begin{enumerate}
 \item[\rm(1)] $S_0=(C_m, \alpha_{i+1}, \alpha_i)$ and $S=(C_n, \alpha_k)$, $2 \leq m \leq n$ and $1 \leq i \leq m-1$ and $n-k=m-i$;
 \item[\rm(2)] $S_0=(C_2, \alpha_2, \alpha_1)$ and $S=(F_4, \alpha_3)$;
 \item[\rm(3)] $S_0 =(B_3, \alpha_2, \alpha_3)$ and $S=(F_4, \alpha_3)$.
 \end{enumerate}

\noindent
In particular, any smooth Schubert varieties of $S$ is linear when $S$ is of type $(F_4, \alpha_4)$.

\end{theorem}

  \begin{theorem}  \label{Theorem 1} Let $S=G/P$ be a rational homogeneous manifold associated to a  short simple root and let $S_0$ be a non-linear smooth Schubert variety of $S$. Then  any subvariety of $S$ having the same homology class as $S_0$ is induced by the action of $\Aut_0(S)$.
 \end{theorem}

For notations, see Section \ref{Section horospherical}. For example, $(C_m, \alpha_{i+1}, \alpha_i)$ denotes the odd symplectic Grassmannian consisting of isotropic $(i+1)$-subspaces of $(\mathbb C^{2m+1}, \omega)$ and $(C_n, \alpha_k)$ denotes the symplectic Grassmannian $Gr_{\omega}(k, \mathbb C^{2n})$ consisting of isotropic $k$-subspaces of $(\mathbb C^{2n}, \omega)$. \\

We remark that  Richmond-Slofstra \cite{RiSl} obtained the same classification of smooth Schubert varieties of rational  homogeneous manifold of Picard number one   by using a combinatorial method  developed by \cite{BiPo} (Grassmnnian Schubert varieties in    \cite{RiSl} are Schubert varieties in rational homogeneous manifolds of Picard number one in our paper). We  reprove  it by a geometric method (Proposition 3.7 of \cite{HoM} and Theorem \ref{Theorem 2}). One   advantage of our geometric method   is that it gives not just classification   but also their rigidity (Theorem \ref{Theorem 1}) at the same time. Moreover, we describe smooth Schubert varieties of rational homogeneous manifolds of Picard number one  geometrically: it is  either a homogeneous submanifold associated to a subdiagram of the Dynkin diagram of $S$, a linear space, or a horospherical variety. This is not true for rational homogeneous manifolds of higher Picard number. For example, odd symplectic flag manifolds (\cite{Mi}) are smooth Schubert varieties of symplectic flag manifolds but they are not horospherical. \\

The remainder of this paper is organized as follows. In Section 2 we give basic definitions and properties of Schubert varieties, horospherical varieties. We also explain our main tool, varieties of minimal rational tangents. 
We will restrict ourselves to the case when $S$ is of type $(F_4, \alpha_3)$ or of type $(F_4, \alpha_4)$. In Section 3   we classify smooth Schubert varieties of the rational homogeneous manifold of type $(F_4, \alpha_3)$ and prove their rigidity, and we complete the proof of Theorem \ref{theorem main}   in the last section by showing that any smooth Schubert varieties of the rational homogeneous manifold of type $(F_4, \alpha_4)$ is linear.

\section{Preliminaries}

\subsection{Schubert varieties}

Let $G$ be a connected semisimple algebraic group over $\mathbb C$.
 Take a Borel subgroup $B$ of $G$ and a maximal torus $T$  in $B$.   Denote by $\Delta^+$ the system of positive roots of $G$ and by $\Phi=\{ \alpha_1, \cdots, \alpha_{\ell}\}$ the system of simple roots of $G$. For a root $\alpha$, write $\alpha = \sum_{i=1}^{\ell} n_i(\alpha)\alpha_i$.
  Let $\frak t$ be the Lie algebra of $T$.
  To each simple root $\alpha_k$ we associate a parabolic subgroup $P$ of $G$, whose Lie algebra $\mathfrak p$ is given by $\mathfrak p = \mathfrak t + \sum_{n_k(\alpha) \geq 0} \frak g_{\alpha}$.
  The reductive part of $\frak p$ is given by $ \mathfrak t + \sum_{n_k(\alpha)=0}\frak g_{\alpha}$  and the nilpotent part   of $\frak p$ is given by $ \sum_{n_k(\alpha) >0}\frak g_{\alpha}$.
 The homogeneous manifold  $S=G/P$ is called the rational homogeneous manifold associated to $\alpha_k$. We will   denote it by $(G, \alpha_k)$.

 Let $\mathcal W$ be the Weyl group of $G$. For $w \in \mathcal W$, set $\Delta(w)=\{ \beta \in \Delta^+  :  w(\beta) \in -\Delta^+\}$.
Define a subset $\mathcal W^P$ of  $\mathcal W$  by
 $\mathcal W^P:=\{ w \in \mathcal W: \Delta(w) \subset \Delta(U_P)\},$   where   $\Delta(U_P)=\{
 \alpha \in \Delta^+: n_{\alpha_k}(\alpha) >0\}$.   Then we have a cell decomposition
  $S=\coprod _{w \in \mathcal W^P} B.x_w  $,
  where $x_w=wP, w \in \mathcal W^P$ are $T$-fixed points in $S$.
  For each $w \in \mathcal W^P$, the closure $S(w)$  of $B. x_w$ is called  {\it the Schubert variety of type} $w$.

 \subsection{Horospherical  varieties}

\label{Section horospherical}

Let $L$ be a connected reductive algebraic group.
Let $H$ be a closed  subgroup of $L$.
A homogeneous space $L/H$ is said to be {\it horospherical} if $H$ contains the unipotent radical of a Borel subgroup of $L$. In this case, the normalizer $N_L(H)$ of $H$ in $L$ is a parabolic subgroup $P$ of $L$ and $P/H$ is a torus $(\mathbb C^{\times})^r$. Thus there is a $(\mathbb C^{\times})^r$-bundle structure on $L/H$ over $L/P$.    A normal $L$-variety is called  {\it horospherical} if it contains an open dense $L$-orbit isomorphic to a horospherical homogeneous space  $L/H$.

For a dominant weight $\varpi$ of $L$ let $V_L(\varpi)$ denote the irreducible representation space of $L$  with highest weight $\varpi$. Fix a Borel subgroup of $L$.
Let $\{\alpha_1, \cdots, \alpha_n\}$ be the system of simple roots of $L$ and let $\{ \varpi_1, \cdots, \varpi_{n}\}$ be the system of fundamental weights  of $L$.   Take a highest weight vector $v_i$ in $V_L(\varpi_i)$ for $i=1, \cdots n$.   Then the $L$-orbit of $[v_i]$ in $\mathbb P(V_L(\varpi_i))$ is the rational homogeneous variety of type $(L, \alpha_i)$.

  For $i \not=j$, the  closure  of the $L$-orbit of $[v_i+v_j]$ in $\mathbb P(V_L(\varpi_i) \oplus V_L(\varpi_j))$  is a horospherial $L$-variety (Proposition 2.1 of \cite{HoH}).
We will denote the closure of $L.[v_i+v_j]$ in $\mathbb P(V_L(\varpi_i) \oplus V_L(\varpi_j))$  by
 $(L, \alpha_i, \alpha_j)$.
It has  three $G$-orbits: one open orbit $L.[v_i + v_j]$ and two closed orbits, $ L.[v_i]$ and $  L.[v_j]$.  For more details on horospherical varieties see \cite{Pa}.

\begin{proposition} [Proposition 1.8 and Proposition 1.9 and Proposition 1.10 of \cite{Pa}]  \label{homogeneous horospherical varieties} $\,$

\begin{enumerate}
\item The horospherical variety $(A_n, \alpha_1, \alpha_n)$ with $n \geq 2$ is isomorphic to the rational homogeneous manifold of type  $   (D_{n+1}, \alpha_1)$.
\item The horospherical variety $(A_n, \alpha_i, \alpha_{i+1})$ with $n \geq 3$ and $1 \leq i \leq n-1$ is isomorphic to the rational homogeneous manifold of type $ (A_{n+1}, \alpha_{i+1})$.
\item The horospherical variety $(D_n, \alpha_{n-1}, \alpha_{n})$ with $n \geq 4$ is isomorphic to the rational homogeneous manifold of type  $(D_{n+1}, \alpha_n)=(B_n, \alpha_n)$.
\end{enumerate}
\end{proposition}

\begin{proposition}[Theorem 0.1 and Theorem 1.7 of  \cite{Pa}]  \label{classification nonhomogeneous}

Let $L$ be a connected reductive algebraic group.  Let $X$ be a smooth projective horospherical  $L$-variety of Picard number one.  
Then $X$ is either homogeneous or one of the following.

\begin{enumerate}
\item [{\rm(1)}] $(B_n, \alpha_{n-1}, \alpha_n)$, $n \geq 3${\rm;}
\item [{\rm(2)}] $(B_3, \alpha_1, \alpha_3)${\rm;}
\item [{\rm(3)}] $(C_n, \alpha_{i+1}, \alpha_i)$, $n\geq 2$ and $i \in \{ 1,2, \cdots, n-1\}${\rm;}
\item [{\rm(4)}] $(F_4, \alpha_2, \alpha_3)${\rm;}
\item [{\rm(5)}] $(G_2, \alpha_2, \alpha_1)$.
\end{enumerate} 
\end{proposition}

In Proposition 4.1 of \cite{HoH}, we describe an equivariant embedding of a smooth horospherical variety of Picard number one into a rational homogeneous manifold of Picard number one as a linear section. Among them $(C_m, \alpha_{i+1}, \alpha_{i})$ is a smooth Schubert variety of $(C_{m+1}, \alpha_{i+1})$. We have two more smooth Schubert varieties as follows.

\begin{proposition}  \label{linear sections} Let $S$ be a rational homogeneous manifold of type $(F_4, \alpha_3)$ and let $S_0$ be one of the following horospherical varieties:

\begin{enumerate}

\item    $S_0=(B_3, \alpha_{2}, \alpha_3)${\rm;}

 \item   $S_0=(C_2, \alpha_{2}, \alpha_1)${\rm;}

\end{enumerate}

 \noindent Then there is an embedding of   $S_0$ into $S$ as a   smooth Schubert variety.
\end{proposition}

\begin{proof}
 We recall how to embed $X=(B_3, \alpha_2, \alpha_3)$ into $S=(F_4, \alpha_3)$. For details see \cite{HoH}. The rational homogeneous manifold $\mathbb S=(B_3, \alpha_3)$ can be embedded into the variety $\mathcal C_x(\mathcal S)$ of minimal rational tangents  of $\mathcal S=(F_4, \alpha_4)$ at $x \in \mathcal S$, and the isotropy group $\mathcal P$ of $\mathcal G=Aut(\mathcal S)$ at $x$  acts transitively on $\mathbb S$. Thus the cone $\widehat{\mathbb S}$ over $\mathbb S$ with vertex $x$  can be embedded into $\mathcal S$ as a linear section and $\mathcal P$ stabilizes $\widehat{\mathbb S}$.  Furthermore, $X$ is the Fano variety  $F_1(\widehat{\mathbb S})$ of lines lying on the cone $\widehat{\mathbb S}$ over $\mathbb S$, and $S$ can be embedded into the Fano variety $F_1(\mathcal S)$ of lines lying on $\mathcal S$. The embedding of $X$ into $S  $ is induced by the embedding of $\widehat{\mathbb S}$ into $\mathcal S$. Therefore, $\mathcal P$ stabilizes $X$.

Since the stabilizer of $X$ in $G=\Aut(S)$ contains a Borel subgroup of $G$ and $X$ is irreducible, $X$ is a Schubert variety (Proposition 2.1 in \cite{HoM}). This completes the proof for the case (1).

For (2) just embed $S_0$ into a rational homogeneous manifold $S_1$ of type $(C_3, \alpha_2)$ and consider the embedding of $S_1$ into $S$ as a homogeneous submanifold associated to a subdiagram of the marked Dynkin diagram of $S$.
\end{proof}


 \subsection{Varieties of minimal rational tangents} \label{section vmrt of Schubert varieties}

Let $X$ be a uniruled projective manifold with an ample line bundle $\mathcal L$. By a (parameterized) {\it rational curve} on $X$ we mean a nonconstant holomorphic map $f: \mathbb P^1 \rightarrow X$.
 A rational curve $f$ is said to be {\it free} if    the pull-back   $f^*TX$ of the tangent bundle $TX$ of $X$  on $\mathbb P^1$ is semipositive. A free rational curve $f$ such that the degree $f^*\mathcal L$ is minimum among all free rational curves is called a {\it minimal rational curve}.
 Let $\mathcal H$ be a connected component of $\text{Hom}(\mathbb P^1, X)$ containing a minimal rational curve and let $\mathcal H^0$  be the subset consisting of free rational curves. The  quotient space $\mathcal K=\mathcal H^{0}/\Aut(\mathbb P^1)$ of (unparameterized) minimal rational curves  is  called 
    a {\it minimal rational component}.

Fix a minimal rational component $\mathcal K$. When we say a minimal rational curve we mean a  rational curve belonging to $\mathcal K$. For a general $x \in X$ the space $\mathcal K_x$ of minimal rational curves passing through $x$ is a projective manifold. Define a rational map $\Psi$ from $\mathcal K_x$ to $\mathbb P(T_xX)$ by sending  a minimal rational curve immersed at $x$ to the tangent line at $x$. The strict transformation $\mathcal C_x(X)$ of $\Psi$ is called the {\it variety of minimal rational tangents of} $X$ {\it at} $x$. The union of $\mathcal C_x(X)$ over general $x \in X$ forms a fiber bundle $\mathcal C(X)$ over $X$. The variety of minimal rational tangents was introduced in \cite{HwM99} to study geometric structures on uniruled projective manifolds. For more details on the variety  of minimal rational tangents and its applications to the study of geometric structures on uniruled projective manifolds, see   \cite{Mk16},  the most recent survey. \\

Let $S=G/P$ be a rational homogeneous manifold associated to a simple root. Then the Picard number of $S$ is one and the ample generator $\mathcal L$ of the Picard group defines a  $G$-equivariant embedding of $S$ into the projective space $\mathbb P(H^0(S, \mathcal L)^*)=\mathbb P^N$. Lines $\mathbb P^1$ in $\mathbb P^N$ lying on $S$ are minimal rational curves, and we will choose the family $\mathcal K$ of lines lying on $S$ as our minimal rational component, so that the variety $\mathcal C_x(S)$ of minimal rational tangents of $S$ at any $x$  in $S$ is defined by the space of all tangent directions of lines lying on $S$ passing through $x$. If $S$ is associated to a long root, then $G$ acts on $\mathcal K$  transitively.  If $S$ is associated to a short root, then $\mathcal K$ has two $G$-orbits. In any case, by a general line we mean a line corresponding to  a point in   the open $G$-orbit  in $\mathcal K$, and by a general point in $\mathcal C_x(S)$ we mean the tangent direction of a general line. 
Let $\mathcal C_x(S)^{gen}$ denote the subvariety of $\mathcal C_x(S)$ consisting of the tangent directions of general lines in $S$. For an explicit description of the variety $\mathcal C_x(S)$ of minimal rational tangents of $S$ and its application to the deformation rigidity of $S$,     see \cite{HwM02},  \cite{HwM04b}, and  \cite{HwM05}.

 Let $S_0$ be a Schubert variety of $S$. By Proposition 3.1 of \cite{HoM}, $S_0$ is covered  by lines of $S$ lying on $S_0$  and  is of Picard number one (the same arguments in the proof work for the case when $S_0$ is singular). Consider the family $\mathcal K_0$ of all lines lying on $S_0$.
 The stabilizer $Stab_G(S_0)$ of $S_0$ in $G$ is a parabolic subgroup of $G$. By a {\it general} point in $S_0$ we mean a point $x$ in the open orbit  of   $Stab_G(S_0)$ in $S_0$. In particular, the base point of $S_0$ is a general point.
For a general point $x$ of $S_0$, define the variety  $\mathcal C_x(S_0)$ of minimal rational tangents of $S_0$ at $x$ by the set of tangents directions of lines lying on $S_0$ passing through $x$. Then $\mathcal C_x(S_0)=\mathcal C_x(S) \cap \mathbb P(T_xS_0)$ (Proposition 3.1 of \cite{HoM}). By a {\it general} point of $\mathcal C_x(S_0)$ we mean a point in  $\mathcal C_x(S_0) \cap \mathcal C_x(S)^{gen}$. \\

Let $x_0$ be the base point of $S$ at which the isotropy group of $G$ is $P$.
Let $B$ be a Borel subgroup of $G$ contained in $P$ and $T$ be a maximal torus of $B$. Let $L$ be the reductive part of $P$ containing $T$.

\begin{proposition} [p.352 of \cite{HoM}, Proposition 4.1 of \cite{HoC}] \label{necessary conditions} Let $S=G/P$ be a rational homogeneous manifold associated to a simple root and let $S_0$ be a  Schubert variety.   Let $x=gx_0$ be a general point of $S_0$ and 
let $(L \cap B)_x$ denote the  conjugate $g(L \cap B)$ of the Borel subgroup $L \cap B$ of $L$.  Then
 \begin{enumerate}

   \item $\mathcal C_x(S_0)$  is invariant under the action of   $(L \cap B)_x$.
   \item If $S_0$ is smooth, then $\mathcal C_x(S_0)$ is smooth and is the closure of a  $(L \cap B)_x$-orbit   in $\mathcal C_x(S)$.
 \end{enumerate}

\end{proposition}

\begin{proof}  For the base point $x_w=w.x_0$ of $S_0$, (1) and (2) follows from the arguments in p.352 of \cite{HoM} or Proposition 4.1 of \cite{HoC}.
It remains to show (1) for a general point $x$ of $S_0$, i.e., for any point in the orbit of  $Stab_G(S_0)$ of $x_w $.  By arguments in p.352 of \cite{HoM}, $\mathcal C_{x_w}(S_0)$ is invariant under the action of $(L\cap B)_{x_w} =w(L \cap B)$. Then $\mathcal C_{gx_w}( gS_0)$ is invariant under the action of $gw(L \cap B)$ for any $g \in G$. In particular, for $b \in Stab_G(S_0)$, $\mathcal C_{bx_w}( S_0)=\mathcal C_{bx_w}(b S_0)$ is invariant under the action of $bw(L \cap B)$. Therefore,     for a general point $x=bx_w$   of $S_0$, $\mathcal C_x(S_0)$  is invariant under the action of   $(L \cap B)_x=bw(L \cap B)$.
\end{proof}

We will consider   the following two conditions (I), (II) on the variety $ \mathcal C_x(S_0)$ of minimal rational tangent of the  `model' Schubert variety $S_0$:
 \begin{enumerate}
 \item [\rm(I)] at a general point $\alpha \in \mathcal C_x(S_0)$, for any $h \in P_x$ sufficiently
close to the identity element $e\in P_x$ and satisfying
$T_{\alpha}\left(h\mathcal C_x(S_0) \right) =
T_{\alpha}\left(\mathcal C_x(S_0) \right)$ we must have $h\mathcal C_x(S_0) = \mathcal C_x(S_0)$;
 \item [\rm(II)] any local deformation    of $ \mathcal C_x(S_0)$ in $\mathcal C_x( S )$ is induced by the action of $P_x$.
 \end{enumerate}

\begin{proposition} [Proposition 3.2 of \cite{HoM}]  \label{general case - inductive step}
Let $S=G/P$ be a rational homogeneous manifold associated to a   simple root, and $S_0  $
be a smooth Schubert variety of $S$. 
Assume that   $\mathcal C_x(S_0)$ satisfies  {\rm(I)} and {\rm(II)}
 at a general point $x \in S_0$.
Then,  the following holds true.
  \begin{enumerate}
 \item[\rm (1)]
    If  a  smooth subvariety  $Z$ of $S$ is uniruled by lines of $S$ lying on $Z$ and  contains $x$ as a general point with  $\mathcal C_x(Z) = \mathcal C_x( S_0)$, then $S_0$ is contained in $Z$.
  \item[\rm (2)]
Any local deformation of $S_0$ in $S$ is induced by the action of $G$.
 \end{enumerate}

\end{proposition}


 \section{$(F_4, \alpha_3)$-case}

 Let $S=G/P$ be the rational homogeneous manifold  of type $(F_4, \alpha_3)$.
 Let $o \in S$ be the base point.  Then
  $\mathcal C_o(S)$ is the projectivization of the cone
 $$\{ e \otimes q + (f\wedge f') \otimes q^2 : e \wedge f \wedge f'=0, e, f, f' \in E,  q \in Q \} $$
 in $(E \otimes Q)  \oplus (\wedge^2 E \otimes S^2 Q)$, where $E$ is a complex vector space of dimension 3 and $Q$ is a complex vector space of dimension 2 (see \cite{HwM04b}).
Via the map $[e \otimes q + (f\wedge f') \otimes q^2]  \in \mathcal C_0(S) \mapsto   [q] \in \mathbb P(Q)$,  $\mathcal C_o(S)$ can be think of as a fiber bundle over $\mathbb P(Q)=\mathbb P^1$ with fiber isomorphic to the smooth quadric $\mathbb Q^4 \subset \mathbb P(E \oplus \wedge^2E)$.
   Let $\rho: P \rightarrow GL(T_{o}S)$ be the isotropy representation.
 Then   $\rho(P)$ is $(SL(E) \times SL(Q))\ltimes (E^* \otimes Q^*)$, where $E^* \otimes Q^*$ acts on $E \otimes Q$ trivially and maps $\wedge^2 E \otimes S^2 Q$ to $E \otimes Q$

   If $S_0$ is  the homogeneous submanifold associated to the subdiagram of of type $(C_3, \alpha_2)$ of $S$, then $\mathcal C_x(S_0)$ is the linear section of $\mathcal C_x(S)$ by $\mathbb P((F_2 \otimes Q) \oplus (F_2^{\perp} \otimes S^2Q))$, where  $F_2$ is a subspace of $E$ of dimension $2$, and  is isomorphic to $\mathbb P(\mathcal O(-1)^2 \oplus \mathcal O(-2))$.

   If $S_0$ is the horospherical variety $(B_3, \alpha_2, \alpha_3)$ in $S$, then $\mathcal C_x(S_0)$ is the linear section of $\mathcal C_x(S)$ by $\mathbb P((F_1 \otimes Q) \oplus (F_1^{\perp} \otimes S^2Q))$, where  $F_1$ is a subspace of $E$ of dimension $1$, and is isomorphic to $\mathbb P(\mathcal O(-1) \oplus \mathcal O(-2)^2)$.

   If $S_0$ is the horospherical variety $(C_2, \alpha_{2}, \alpha_1)$ in $S$, then $\mathcal C_x(S_0)$ is the linear section of $\mathcal C_x(S)$ by $\mathbb P((e \otimes Q) \oplus (f^* \otimes S^2Q) )$, where $e \in E$ and $f^* \in E^*$ be such that $\langle e, f^* \rangle =0$, and    is isomorphic to $\mathbb P(\mathcal O(-1) \oplus \mathcal O(-2))$.

\begin{lemma} \label{classification in Q4}
Let $B^1$ be a Borel subgroup of $SL(E)$.
 The smooth closures of $B^1$-orbits in $\mathbb Q^4 \subset \mathbb P(E \oplus \wedge^2E) \simeq \mathbb P(E \oplus E^*)$ intersecting the open $SL(E)$-orbit are given by
$$\mathbb Q^4, \mathbb P(F_1 \oplus F_1^{\perp}),\, \mathbb P(F_2 \oplus F_2^{\perp}),\, \mathbb P(V_1 \oplus W_1)$$
where $F_i$ ($i=1,2$) is a subspace of $E$ of dimension $i$ and $F_i^{\perp}$ is the annihilator of $F_i$, and $V_1$ is a subspace of $E$ of dimension one and $W_1$ is a subspace of $V_1^{\perp}$ of dimension one.

\end{lemma}

 \begin{proof}
 Take a basis $\{e_1, e_2, e_3 \}$ of $E$ compatible with $B^1$. 
 Let $\tilde E$ be a vector space of dimension 4 containing $E$. Extend $\{e_1, e_2, e_3\}$ to a basis $\{e_1, e_2, e_3, e_4\}$ of $\tilde E$.
 Recall that the isomorphism $E \oplus \wedge^2 E \rightarrow \wedge^2 \tilde E$ is given by $e + f \wedge f' \mapsto e\wedge  e_4 + f \wedge f'$ and, under this isomorphism,  the closure of $SL(E).[e_1 + e_1 \wedge e_2]$ in $\mathbb P(E \oplus (\wedge ^2E))$ is isomorphic to $G(2,4) \simeq  \mathbb Q^4 \subset \mathbb P(\wedge^2 \tilde E) \simeq \mathbb P^5$ (Proposition \ref{homogeneous horospherical varieties}). Identifying  $\wedge^2 E$ with $E^*$ and considering quadratic form on $E \oplus E^*$, we can see that the closure of $SL(E).[e_1 +e_3^*]$ in $\mathbb P(E \oplus E^*)$  is $\mathbb Q^4 \subset \mathbb P^5$.

  Now $\mathbb Q^4$ has three $SL(E)$-orbits, $\mathbb P(E)$, $\mathbb P(  E^*)$ and the open orbit $\mathcal O$.
  The closures of $B^1$-orbits in $\mathbb Q^4$ which intersect the open orbit $\mathcal O$  are
 \begin{enumerate}
 \item[(a)] $ \cl(B_1.(e_1+e_3^*))=\mathbb Q^4$
 \item[(b)] $\cl(B_1.(e_1+ e_2^*))$,  $\cl(B_1.(e_2+e_3^*))$  (3-dimensional and singular)
 \item[(c)] $\cl(B_1.(e_2 +e_1^*))=\mathbb P^2$, $\cl(B_1.(e_3+e_2^*))=\mathbb P^2$
 \item[(d)] $\cl(B_1.(e_3+e_1^*))=\mathbb P^1$
 \end{enumerate}
We may express $\mathbb P^2$'s in (c) as $\mathbb P(F_2 \oplus F_2^{\perp})$ and $\mathbb P(F_1 \oplus F_1^{\perp})$, where $F_i $  is a subspace of $E$ of dimension $i$ and $F_i^{\perp}$ is the annihilator of $F_i$ for $i=1,2$.
 \end{proof}

The space $\mathcal K$ of $\mathbb P^2$'s in $\mathbb Q^4 \subset \mathbb P(E \oplus \wedge^2E)\simeq \mathbb P(E \oplus E^*)$ has two connected components, $\mathcal K_1$ and $\mathcal K_2$,  each of which is isomorphic to  $\mathbb P^3$. One of them contains $\mathbb P(E)$, and the other contains $\mathbb P(E^*)$.

\begin{lemma} \label{deformtion of P2 in Q4}
Each connected component of the space of $\mathbb P^2$'s in $\mathbb Q^4 \subset \mathbb P(E \oplus E^*)$ has two $SL(E) \ltimes E^*$-orbits: one is closed and the other is open.
\end{lemma}

\begin{proof}

Let $\mathbb P(F)$ be a $\mathbb P^2$ contained in $\mathbb Q^4$ in the same connected component as $\mathbb P(E)$. If $\mathbb P(F) \not= \mathbb P(E)$, then we have $\dim (\mathbb P(F) \cap \mathbb P(E))=0$, and ($\mathbb P(F) \cap  \mathbb P(E^*)$ has dimension 1 or is empty). In the first case, we have
$$F=F_1 \oplus F_1^{\perp}$$ for some subspace $F_1 \subset E$ of dimension 1.
In the second case, there is a linear map $\varphi: E \rightarrow \wedge^2 E$ such that $\dim \text{Ker}\, \varphi =1$ and $F=F_{\varphi}$, where $F_{\varphi} \subset E$ is the graph of $\varphi$. Let $\varphi_1:E \rightarrow  \wedge^2 E$ be a linear map defined  by $\varphi_1(e) = e_1 \wedge e$, where $e_1$ is a basis of $\text{Ker}\, \varphi $.  The condition $e \wedge \varphi(e) =0$ for any $e \in E$ implies that $\varphi $ is $\lambda \varphi_1$ for some $\lambda \in \mathbb C^{\times} =\mathbb C -\{0\}$. To see this, extend $\{e_1\}$ to a basis $\{e_1, e_2, e_3\}$ of $E$ and  write
\begin{eqnarray*}
\varphi(e_2)&=& \varphi^2_{12}e_1 \wedge e_2 + \varphi^2_{23}e_2 \wedge e_3 + \varphi^2_{31} e_3 \wedge e_1 \\
\varphi(e_3)&=& \varphi^3_{12}e_1 \wedge e_2 + \varphi^3_{23}e_2 \wedge e_3 + \varphi^3_{31} e_3 \wedge e_1.
\end{eqnarray*}
 From $0=e_2 \wedge \varphi(e_2)=e_3 \wedge \varphi(e_3)$ it follows that $\varphi^2_{31}=\varphi^3_{12}=0$. From $0=(e_1+ e_2) \wedge(\varphi(e_1) + \varphi(e_2))=e_1 \wedge \varphi(e_2) + e_2 \wedge \varphi(e_2) = e_1 \wedge \varphi(e_2)$
 it follows that $\varphi^2_{23}=0$. Similarly, $\varphi^3_{23}=0$.
 From $0=(e_2 + e_3) \wedge (\varphi(e_2) + \varphi(e_3))=e_2 \wedge \varphi(e_3) + e_3 \wedge \varphi(e_2)$, it follows that $e_2 \wedge (\varphi^3_{31}e_3 \wedge e_1) + e_3 \wedge (\varphi^2_{12} e_1 \wedge e_2)=0$ and thus $\varphi^3_{31}=-\varphi^2_{12}$. Put $\lambda:=\varphi^2_{12}$. Then $\varphi = \lambda \varphi_1$ and
   $F_{\varphi}$ is spanned by $$e_1,\,\, e_2 + \lambda e_1 \wedge e_2,\,\, e_3 -\lambda e_3 \wedge e_1.$$
   We remark that $\lim_{\lambda \rightarrow 0}F_{\lambda \varphi_1} =E$ and $\lim_{\lambda \rightarrow \infty} F_{\lambda \varphi_1} = F_1 \oplus F_1^{\perp}$, where $F_1 =\text{Ker}\, \varphi_1$.

 The action of $E^*$ on  $E \oplus \wedge^2 E$ is given by zero on $E$ and  by the interior product on $\wedge^2E$. For example, $ce_1^*.(e_1 \wedge e_2) = ce_2 + e_1 \wedge e_2$ and $ce_1^*.(e_3 \wedge e_1) = - ce_3 + e_3 \wedge e_1$, where $c \in \mathbb C$. Hence, there is an element $e \in E^*$ which maps $\mathbb P(F_1 \oplus F_1^{\perp})$ to $\mathbb P(F_{\varphi})$, while $\mathbb P(E)$ is fixed by the action of $SL(E) \ltimes E^*$.  Therefore, the connected component of the  space of $\mathbb P^2$'s in $\mathbb Q^4$ containing $\mathbb P(E)$ has two $SL(E) \ltimes E^*$-orbits, the orbit of $\mathbb P(E)$ (which is a one point set $\{\mathbb P(E) \}$) and the orbit of $\mathbb P(F_1 \oplus F_1^{\perp})$, where $F_1$ is a subspace of $E$ of dimension one. The first one is closed and the second one is open.  \\


Let $\mathbb P(F)$ be a $\mathbb P^2$ contained in $\mathbb Q^4$ in the same connected component as $\mathbb P(E^*)$. If $\mathbb P(F) \not=\mathbb P(E^*)$, by the same arguments as in the previous case, $F$ is either $F_2 \oplus F_2^{\perp}$ for some subspace $F_2$ of $E$ of dimension $2$ or is  spanned by
$$e_2 \wedge e_3 + \lambda e_2,\,\, e_3 \wedge e_1 -\lambda e_1,\,\, e_1 \wedge e_2$$
for some basis $\{e_1, e_2, e_3\}$ of $E$.
 Subspaces $F$ of the first form are fixed by the action of  $E^*$.
For each subspace $F$ of the second form, there is an element $e^* \in E^*$ which maps $E^* $ to $\langle e_2 \wedge e_3 + \lambda e_2, e_3 \wedge e_1 -\lambda e_1, e_1 \wedge e_2 \rangle$  (just think of $E^*$ as $\langle e_2\wedge e_3, e_3\wedge e_1, e_1 \wedge e_2 \rangle$). Therefore, the connected component of the  space of $\mathbb P^2$'s in $\mathbb Q^4$ containing $\mathbb P(E^*)$ has two $SL(E) \ltimes E^*$-orbits, the orbit of $\mathbb P(E^*)$ and the orbit of $\mathbb P(F_2 \oplus F_2^{\perp})$, where $F_2$ is a subspace of $E$ of dimension two. The first one is open and the second one is closed.
\end{proof}

\noindent {\bf Remark.}  Let $\mathcal Y_0 =\mathbb P(V_1 \oplus W_1)$ where $V_1 \subset E$ is a subspace of dimension one and $W_1 \subset V_1^{\perp}$ is a subspace of dimension one. Since any line $\mathbb P^1$ in $\mathbb Q^4$ is the intersection of two $\mathbb P^2$'s, contained in  different connected components of $\mathcal K$, any local deformation $\mathcal Y_t$ of $\mathcal Y_0$ is the intersection $\mathcal X_{1,t} \cap \mathcal X_{2,t}$, where $\mathcal X_{i,t}$ belongs to  in $\mathcal K_i$ for $i=1,2$.  In the proof of Proposition \ref{deformtion of P2 in Q4}, we prove that,  up to the action of $SL(E) \ltimes E^*$, $\mathcal X_{1,t}=\mathbb P(V_1 \oplus V_1^{\perp})$. Since $\mathcal Y_t=\mathcal X_{1,t} \cap \mathcal X_{2,t}$ is $\mathbb P^1$, $\mathcal X_{2,t}$ is of the form $\mathbb P(F_{2,t} \oplus F_{2,t}^{\perp})$, where $F_{2,t}$ is a subspace of $E$ of dimension two, and $\mathcal Y_t$ is of the form $\mathbb P(V_1 \oplus W_{1,t})$, where $W_{1,t} $ is a subspace of $ V_1^{\perp}$ of dimension one. Therefore, up to the action of $SL(E) \ltimes E^*$ again, $\mathcal Y_t =\mathcal Y_0$. \\


 \begin{proposition} \label{classification of vmrt}
 Let $S=G/P$ be the rational homogeneous manifold  of type $(F_4, \alpha_3)$.
 Varieties of minimal rational tangents of smooth Schubert varieties   of $S$ are   of the following forms: \\

 \begin{tabular} {|c|c|}
 \hline
 $\mathcal C_o(S) \cap \mathbb P((E\otimes q) \oplus (E^* \otimes q^2))$ \,\,&\,\, $\mathcal C_o(S)$ \\ [3pt]
 \hline
   $\mathbb P(E \otimes q)$  \,\,&\,\, $\mathcal C_o(S) \cap \mathbb P((F_1 \otimes Q) \oplus (F_1^{\perp}\otimes S^2Q))$    \\[3pt]
   $\mathbb P((F_2 \otimes q) \oplus (F_2^{\perp}\otimes q^2))$ & $\mathcal C_o(S) \cap \mathbb P((F_2 \otimes Q)\oplus (F_2^{\perp}\otimes S^2Q))$ \\[3pt]
 \hline
  $\mathbb P(F_2 \otimes q)$ \,\,&\,\, $\mathcal C_o(S) \cap \mathbb P((V_1 \otimes Q) \oplus (W_1\otimes S^2Q))$ \\[3pt]
 \hline
   $\mathbb P(F_1 \otimes q)$ \,\,& \,\,$\mathbb P(e \otimes Q)$ \\[3pt]
  \hline
 \end{tabular}
 %

 \vskip 10 pt \noindent
 where $F_i$ is a subspace of $E$ of dimension $i$ for $i=1,2$ and $V_1$ is a subspace of $E$ of dimension one and $W_1$ is a subspace of $V_1^{\perp}$ of dimension one and $e \in E$ and $q \in Q$.

 The corresponding smooth Schubert varieties are \\

 \begin{tabular}{|c|c|}
 \hline
   $(B_3, \alpha_3)$ \,\,&\,\, $S$ \\[3pt]
 \hline
    $\mathbb P^3_{A_3}$  \,\,&\,\, $(C_3,\alpha_2)$   \\[3pt]
      $ (B_2, \alpha_2)$  \,\,&\,\, $(B_3, \alpha_2, \alpha_3)$  \\[3pt]
 \hline
   $\mathbb P^2_{A_2}$ \,\,&\,\, $(C_2, \alpha_2, \alpha_1)$\\[3pt]
 \hline
  $ (A_1, \alpha_1)$ \,\,&\,\, $ (A_2, \alpha_1)$ \\[3pt]
 \hline
  \end{tabular}

 \vskip 10 pt \noindent
where $(L, \alpha_i)$ denotes the homogeneous submanifold of $S$  associated to a subdiagram of type $(L, \alpha_i)$, and $(L, \alpha_i, \alpha_j)$ denotes the  horospherical variety embedded in $S$ as in Proposition \ref{linear sections}, and $\mathbb P_{A_3}^3$ and $\mathbb P_{A_2}^2$ denote $\mathbb P^3$ and $\mathbb P^2$ in $(B_3, \alpha_3)$ which are not associated to a subdiagram of the Dynkin diagram of $S$.
 \end{proposition}

 \begin{proof} Recall that the variety
  $\mathcal C_o(S)$ of minimal rational tangents of $S=G/P$ is the projectivization of the cone
 $$\{ e \otimes q + f^* \otimes q^2 : \langle e , f^*\rangle=0, e \in E, f^* \in E^*,  q \in Q \} $$
 in $(E \otimes Q)  \oplus (  E^* \otimes S^2 Q)$, where $E$ is a complex vector space of dimension 3 and $Q$ is a complex vector space of dimension 2, and that the fiber over $[q] \in \mathbb P(Q)$ is $\{ e \otimes q + f^* \otimes q^2 : \langle e ,f^*\rangle=0, e \in E, f^* \in E^*\} \simeq \mathbb Q^4$.
 The semisimple part $L = L^1 \times L^2$ of $P$  is  $SL(E) \times SL(Q)$ and $L \cap B$ is $B^1 \times B^2$, where $B^1$ is a Borel subgroup of $SL(E)$ and $B^2$ is a Borel subgroup of $SL(Q)$.

 Let $S_0$ be a smooth Schubert variety of $S $ and $w \in \mathcal W^P$ be the element corresponding to $S_0$, i.e., $S_0$ is the closure of the $B$-orbit $B.x$ at $x=w.o$.  By  Proposition \ref{necessary conditions}  $\mathcal C_o(w^{-1}S_0)$ is the closure of a  $B^1 \times B^2$-orbit  $B^1 \times B^2 (e \otimes q + f^* \otimes q^2)$, where $(e, f^*) \in E \oplus E^*$ satisfies $\langle e, f^* \rangle=0$ and $q\in Q$. \\ 

 \noindent {\bf Case 1.}  If $S_0$ does not have a general line, then $\mathcal C_o(w^{-1}S_0)$ is contained in $\mathcal C_o(S) \backslash \mathcal C_0(S)^{gen}$, and thus it is contained in $\mathbb P (E \otimes Q)$.  Therefore, $\mathcal C_o(w^{-1}S_0)$ is one of the followings: $\mathbb P(E \otimes q)$, $\mathbb P(F_2 \otimes q)$, $\mathbb P(F_1 \otimes q)$, $\mathbb P(e \otimes Q)$. \\

\noindent {\bf Case 2.} If $S_0$ has a general line, then $\mathcal C_o(w^{-1}S_0)$ intersects  $\mathcal C_o(S)^{gen}$ nontrivially.
 By Lemma \ref{classification in Q4}  the smooth closure  of a $B^1 \times B^2$-orbit  in $\mathcal C_o(S)$ which is a linear section of $\mathcal C_o(S)$ and  intersects $\mathcal C_o(S)^{gen}$, is one of the followings: \\
%

\begin{tabular} {|c|c|}
 \hline
 $\mathcal C_o(S) \cap \mathbb P((E\otimes q) \oplus (E^* \otimes q^2))$ \,\,&\,\, $\mathcal C_o(S)$ \\ [3pt]
 \hline
   $\mathbb P((F_1 \otimes q) \oplus (F_1^{\perp} \otimes q^2))$  \,\,&\,\, $\mathcal C_o(S) \cap \mathbb P((F_1 \otimes Q) \oplus (F_1^{\perp}\otimes S^2Q))$    \\[3pt]
   $\mathbb P((F_2 \otimes q) \oplus (F_2^{\perp}\otimes q^2))$ & $\mathcal C_o(S) \cap \mathbb P((F_2 \otimes Q)\oplus (F_2^{\perp}\otimes S^2Q))$ \\[3pt]
 \hline
  $\mathbb P((V_1 \otimes q)\oplus (W_1 \otimes q^2))$ \,\,&\,\, $\mathcal C_o(S) \cap \mathbb P((V_1 \otimes Q) \oplus (W_1\otimes S^2Q))$ \\[3pt]
 \hline
 \end{tabular}

 \vskip 10 pt \noindent
 where $F_i$ is a subspace of $E$ of dimension $i$ for $i=1,2$ and $V_1$ is a subspace of $E$ of dimension one and $W_1$ is a subspace of $V_1^{\perp}$ of dimension one and $e \in E$ and $q \in Q$.
(Note that $\mathcal C_o(w^{-1}S_0)$ cannot be contained in $\mathbb P(E^* \otimes S^2Q)$.)

Among them, the $P$-orbits of
$\mathbb P((F_1 \otimes q) \oplus (F_1^{\perp} \otimes q^2))$  and $\mathbb P((V_1 \otimes q)\oplus (W_1 \otimes q^2))$ are not closed (see the proof of Proposition \ref{deformtion of P2 in Q4}), so that they cannot be the varieties of minimal rational tangents of Schubert varieties. \\

Combining lists in Case 1 and in  Case 2, we get the desired list.
 \end{proof}

\begin{proposition} \label{F4alpha3} Let   $S$ be the rational homogeneous manifold of type $(F_4, \alpha_3)$ and let $S_0$ be either $(C_2, \alpha_2, \alpha_1)$ or $(B_3, \alpha_2, \alpha_3)$. Then $\mathcal C_x(S_0)$ at a general point  $x \in S_0$  satisfies {\rm(I)} and {\rm(II)} in Proposition \ref{general case - inductive step}.

\end{proposition}

\begin{proof} We will use the same notations as in Proposition \ref{classification of vmrt}.
Assume that $S_0$ is $(B_3, \alpha_2, \alpha_3)$. Then $\mathcal Z_0:=\mathcal C_x(S_0)$ is the linear section of $\mathcal Z:=\mathcal C_x(S)$ by $\mathbb P((F_1 \otimes Q) \oplus (F_1^{\perp} \otimes S^2Q)$ and thus $\mathcal Z_0$ is   the projectivization $\mathbb P(\mathcal F)$ of the vector bundle $\mathcal F$ of  rank 3 over $\mathbb P(Q)$, whose fiber over $[q] \in \mathbb P(Q)$ is $(F_1 \otimes q) \oplus (F_1^{\perp} \otimes q^2)$. Hence $\mathcal F$ is isomorphic to $\mathcal O(-1) \oplus \mathcal O(-2)^2$ over $\mathbb P^1$.

Any local deformation $\mathbb P(\mathcal F_t)$ of $\mathbb P(\mathcal F)$ is also isomorphic to $\mathbb P(\mathcal O(-1) \oplus \mathcal O(-2)^2)$, so that there is a subbundle $\mathcal F_{1,t} \subset \mathcal F_t$ such that $\mathcal F_{1,t} \otimes \mathcal O(1)$ is a trivial vector bundle of rank one. Then there is a subspace $F_{1,t} \subset E$ of dimension one such that  the fiber of $\mathcal F_{1,t}$ at $[q] \in \mathbb P(Q)$ is $F_{1,t} \otimes q$. By acting an element of $SL(E)$, we may assume that $F_{1,t}=F_1$.

By the proof of Lemma \ref{deformtion of P2 in Q4}, the fiber of $\mathcal F_{t}$ at $[q] \in \mathbb P(Q)$ is the graph $F_{\lambda \varphi_1}$ of $\lambda \varphi_1:E \rightarrow E^*$ for some $\lambda \not=0\in \mathbb C \cup \{\infty\}$, depending on $[q]$. Since the assignment $[q] \in \mathbb P(Q) \mapsto \lambda \in (\mathbb C -\{0\})\cup \{\infty\}$ is holomorphic, it is constant. Consequently, $\mathcal F_t$ is $\mathcal F$ up to the action of $(SL(E) \times SL(Q))\ltimes (E^* \otimes Q^*)$.

For $h \in (SL(E) \times SL(Q))\ltimes (E^* \otimes Q^*)$ having nontrivial factor in $E^* \otimes Q^*$, there is a nonzero linear function $\varphi:E \rightarrow E^*$ such that  $h\mathbb P(\mathcal F)=\mathbb P(\mathcal F_{\varphi})$, where $\mathcal F_{\varphi}$ is the vector bundle over $\mathbb P(Q)$ whose fiber at $[q] \in \mathbb P(Q)$ is $\{e \otimes q + \varphi(e) \otimes q^*: e \in E\}$. Then  any point  in $ \mathbb P(\mathcal F) \cap h\mathbb P(\mathcal F)$ is non-generic. If $h$ is in $SL(E) \times SL(Q)$ and $ \mathbb P(\mathcal F)$ is   tangent to $h\mathbb P(\mathcal F))$ at $\alpha \in  \mathbb P(\mathcal F) \cap h\mathbb P(\mathcal F))$, then $h\mathbb P(\mathcal F) = \mathbb P(\mathcal F)$.

By a similar arguments we get the desired results when  $S_0$ is $(C_2, \alpha_2, \alpha_1)$.
\end{proof}

 \begin{proposition} \label{classification F4 alpha3} Let $S=G/P$ be the rational homogeneous manifold  of type $(F_4, \alpha_3)$. Then a  smooth Schubert variety of $S$ is one of the following:
  \begin{enumerate}
  \item   a  homogeneous submanifold associated to  a subdiagram  of the marked Dynkin diagram of $S$;
  \item   a linear space;
  \item   $(B_3, \alpha_2, \alpha_3)$ or $(C_2, \alpha_1, \alpha_2)$, embedded as in Proposition \ref{linear sections}.
  \end{enumerate}
 \end{proposition}

\begin{proof} Proposition \ref{classification of vmrt} and
 Proposition  \ref{F4alpha3} and Proposition \ref{general case - inductive step}.
\end{proof}

\begin{proposition} \label{rigidity F4 alpha3}
Let   $S$ be the rational homogeneous manifold of type $(F_4, \alpha_3)$ and let $S_0$ be either $(C_2, \alpha_2, \alpha_1)$ or $(B_3, \alpha_2, \alpha_3)$. Then any local deformation of $S_0$ in $S$ is induced by the action of $G$.
\end{proposition}

\begin{proof}
By   Proposition \ref{general case - inductive step} and Proposition \ref{F4alpha3}, any local deformation of $S_0$ in $S$ is induced by the action of $G$.
 \end{proof}

\section{$(F_4, \alpha_4)$-case}

  In this section we will consider the case when $S$ is the rational homogeneous manifold of type $(F_4, \alpha_4)$ and prove  that any smooth Schubert variety of rational homogeneous manifold $S$ of type $(F_4, \alpha_4)$ is linear (Proposition \ref{classification F4 alpha4}). We will use that $S$ is a hyperplane section of another rational homogenoeus manifold $S'$ of Picard number one, which is associated to a long simple roots, and that any smooth Schubert variety of $S'$ is a homogeneous submanifold associated to a subdiagram of the marked Dynkin diagram of $S'$ (Proposition 3.7 of \cite{HoM}).  \\

Let $G$ be the simple group of type $F_4$ and let $W$ be the irreducible $G$-representation space of highest weight $\varpi_4$ and let $w_4$ be a highest weight vector in $W$. Then  the $G$-orbit of  $x_0:=[w_4]$ in $\mathbb P(W)$ is   the rational homogeneous manifold $S=G/P$ of type $(F_4, \alpha_4)$.
Let $G'$ be the simple Lie group of type $E_6$ and
  let $W'$ be the irreducible $E_6$-representation space of highest weight $\varpi_6$ and let $w_6'$ be a highest weight vector in $W'$. Then the $G'$-orbit of $x_0':=[w'_6] \in \mathbb P(W')$ is the rational homogeneous manifold $S'=G'/P'$ of type $(E_6, \alpha_6)$.
$W$ can be embedded into $W'$ equivariantly as a hyperplane with $x_0=x_0'$ and  $S=G/P$ is the hyperplane section of   $S'=G'/P'$   by $\mathbb P(W)$. Here, we use the same notation for the fundamental weights $\varpi_1, \dots, \varpi_4$ of $G$ of type $F_4$ and the fundamental weights $\varpi_1, \dots, \varpi_6$ of $G'$ of type $E_6$, for the simplicity of notations. We will adapt the same convention afterwards as long as it does not make any confusion.

 For $w \in \mathcal W^P$,  let $S(w)$ be the closure of $B$-orbit $B.x_w$ in $S$, and
for $w' \in \mathcal W^{P'}$, let $S'(w')$ be the closure of $B'$-orbit $B'.x_{w'}$ in $S'$, where $x_{w'}:=w'.x_0$.
The inclusion map $\mathcal W_G \hookrightarrow \mathcal W_{G'}$ from the Weyl group of $G$ to the Weyl group of $G'$ induces an injective map $$\mathcal W^P=\mathcal W_G/\mathcal W_P \hookrightarrow \mathcal W^{P'}=\mathcal W_{G'}/\mathcal W_{P'}$$
\noindent (Figure A and Figure B).
Thus for $w \in \mathcal W^{P}$, $B.x_w $ is contained in $B'.x_{w'}$ for a unique $w' \in \mathcal W^{P'}$. Then we have either $B.x_w =B'.x_{w'} \subset \mathbb P(W)$ or $B.x_w  \subsetneq B'.x_{w'} $ and  $B.x_w = B'.x_{w'} \cap \mathbb P(W)$, so that  we have either $S(w) = S'(w') \subset \mathbb P(W)$ or $S(w)  \subsetneq S(w')$ and $S(w) = S'(w') \cap \mathbb P(W)$. In any case we have $S(w) =S(w') \cap \mathbb P(W)$.

\begin{figure}[p]
\begin{subfigure}[b]{.4\textwidth}
\begin{tikzpicture}
   \dynkinnode{0}{0}{$\tilde{S}_1$}
   \dynkinarrow{0}{0.4}{0}{0.8}
   \dynkinnode{-0.5}{0.6}{$s_{\tilde{\alpha}_1}$}
   \dynkinnode{0}{1.2}{$\tilde{S}_2$}
   \dynkinarrow{0}{1.6}{0}{2}
   \dynkinnode{-0.5}{1.8}{$s_{\tilde{\alpha}_3}$}
   \dynkinnode{0}{2.4}{$\tilde{S}_3$}
   \dynkinarrow{0}{2.8}{0}{3.2}
   \dynkinnode{-0.5}{3}{$s_{\tilde{\alpha}_4}$}
   \dynkinnode{0}{3.6}{$\tilde{S}_4$}
   \dynkinarrow{-0.2}{4}{-0.8}{4.4}
   \dynkinarrow{0.2}{4}{0.8}{4.4}
   \dynkinnode{-0.8}{4}{$s_{\tilde{\alpha}_2}$}
   \dynkinnode{0.8}{4}{$s_{\tilde{\alpha}_5}$}
   \dynkinnode{-1}{4.8}{$\tilde{S}_5$}
   \dynkinnode{1}{4.8}{$\tilde{S}_6$}
   \dynkinarrow{-1}{5.2}{-1}{5.6}
   \dynkinarrow{1}{5.2}{1}{5.6}
   \dynkinarrow{0.6}{5}{-0.6}{5.8}
   \dynkinnode{-1.5}{5.4}{$s_{\tilde{\alpha}_5}$}
   \dynkinnode{1.5}{5.4}{$s_{\tilde{s}_6}$}
   \dynkinnode{0}{5}{$s_{\tilde{\alpha}_2}$}
   \dynkinnode{-1}{6}{$\tilde{S}_7$}
   \dynkinnode{1}{6}{$\tilde{S}_8$}
   \dynkinarrow{-1}{6.4}{-1}{6.8}
   \dynkinarrow{1}{6.4}{1}{6.8}
   \dynkinarrow{-0.6}{6.2}{0.6}{7}
   \dynkinnode{-1.5}{6.6}{$s_{\tilde{\alpha}_4}$}
   \dynkinnode{1.5}{6.6}{$s_{\tilde{\alpha}_2}$}
   \dynkinnode{0}{6.2}{$s_{\tilde{s}_6}$}
   \dynkinnode{-1}{7.2}{$\tilde{S}_9$}
   \dynkinnode{1}{7.2}{$\tilde{S}_{10}$}
   \dynkinarrow{-1}{7.6}{-1}{8}
   \dynkinarrow{1}{7.6}{1}{8}
   \dynkinarrow{-0.6}{7.4}{0.6}{8.2}
   \dynkinnode{-1.5}{7.8}{$s_{\tilde{\alpha}_3}$}
   \dynkinnode{1.5}{7.8}{$s_{\tilde{\alpha}_4}$}
   \dynkinnode{0}{7.4}{$s_{\tilde{s}_6}$}
   \dynkinnode{-1}{8.4}{$\tilde{S}_{11}$}
   \dynkinnode{1}{8.4}{$\tilde{S}_{12}$}
   \dynkinarrow{-1.2}{8.8}{-1.8}{9.2}
   \dynkinarrow{-0.8}{8.8}{-0.2}{9.2}
   \dynkinarrow{0.8}{8.8}{0.2}{9.2}
   \dynkinarrow{1.2}{8.8}{1.8}{9.2}
   \dynkinnode{-1.8}{8.8}{$s_{\tilde{\alpha}_1}$}
   \dynkinnode{-0.4}{8.8}{$s_{\tilde{s}_6}$}
   \dynkinnode{0.4}{8.8}{$s_{\tilde{\alpha}_3}$}
   \dynkinnode{1.8}{8.8}{$s_{\tilde{\alpha}_5}$}
   \dynkinnode{-2}{9.6}{$\tilde{S}_{13}$}
   \dynkinnode{0}{9.6}{$\tilde{S}_{14}$}
   \dynkinnode{2}{9.6}{$\tilde{S}_{15}$}
   \dynkinarrow{-1.8}{10}{-1.2}{10.4}
   \dynkinarrow{-0.2}{10}{-0.8}{10.4}
   \dynkinarrow{0.2}{10}{0.8}{10.4}
   \dynkinarrow{1.8}{10}{1.2}{10.4}
   \dynkinnode{-1.8}{10.4}{$s_{\tilde{s}_6}$}
   \dynkinnode{-0.4}{10.4}{$s_{\tilde{\alpha}_1}$}
   \dynkinnode{0.4}{10.4}{$s_{\tilde{\alpha}_5}$}
   \dynkinnode{1.8}{10.4}{$s_{\tilde{\alpha}_3}$}
   \dynkinnode{-1}{10.8}{$\tilde{S}_{16}$}
   \dynkinnode{1}{10.8}{$\tilde{S}_{17}$}
   \dynkinarrow{-1}{11.2}{-1}{11.6}
   \dynkinarrow{1}{11.2}{1}{11.6}
   \dynkinarrow{0.6}{11}{-0.6}{11.8}
   \dynkinnode{-1.5}{11.4}{$s_{\tilde{\alpha}_5}$}
   \dynkinnode{1.5}{11.4}{$s_{\tilde{\alpha}_4}$}
   \dynkinnode{0}{11}{$s_{\tilde{\alpha}_1}$}
   \dynkinnode{-1}{12}{$\tilde{S}_{18}$}
   \dynkinnode{1}{12}{$\tilde{S}_{19}$}
   \dynkinarrow{-1}{12.4}{-1}{12.8}
   \dynkinarrow{1}{12.4}{1}{12.8}
   \dynkinarrow{0.6}{12.2}{-0.6}{13}
   \dynkinnode{-1.5}{12.6}{$s_{\tilde{\alpha}_4}$}
   \dynkinnode{1.5}{12.6}{$s_{\tilde{\alpha}_2}$}
   \dynkinnode{0}{12.2}{$s_{\tilde{\alpha}_1}$}
   \dynkinnode{-1}{13.2}{$\tilde{S}_{20}$}
   \dynkinnode{1}{13.2}{$\tilde{S}_{21}$}
   \dynkinarrow{-1}{13.6}{-1}{14}
   \dynkinarrow{1}{13.6}{1}{14}
   \dynkinarrow{-0.6}{13.4}{0.6}{14.2}
   \dynkinnode{-1.5}{13.8}{$s_{\tilde{\alpha}_3}$}
   \dynkinnode{1.5}{13.8}{$s_{\tilde{\alpha}_1}$}
   \dynkinnode{0}{13.4}{$s_{\tilde{\alpha}_2}$}
   \dynkinnode{-1}{14.4}{$\tilde{S}_{23}$}
   \dynkinnode{1}{14.4}{$\tilde{S}_{22}$}
   \dynkinarrow{-0.8}{14.8}{-0.2}{15.2}
   \dynkinarrow{0.8}{14.8}{0.2}{15.2}
   \dynkinnode{-0.8}{15.2}{$s_{\tilde{\alpha}_2}$}
   \dynkinnode{0.8}{15.2}{$s_{\tilde{\alpha}_3}$}
   \dynkinnode{0}{15.6}{$\tilde{S}_{24}$}
   \dynkinarrow{0}{16}{0}{16.4}
   \dynkinnode{-0.5}{16.2}{$s_{\tilde{\alpha}_4}$}
   \dynkinnode{0}{16.8}{$\tilde{S}_{25}$}
   \dynkinarrow{0}{17.2}{0}{17.6}
   \dynkinnode{-0.5}{17.4}{$s_{\tilde{\alpha}_5}$}
   \dynkinnode{0}{18}{$\tilde{S}_{26}$}
   \dynkinarrow{0}{18.4}{0}{18.8}
   \dynkinnode{-0.5}{18.6}{$s_{\tilde{s}_6}$}
   \dynkinnode{0}{19.2}{$\tilde{S}_{27}$}
\end{tikzpicture}
\caption{[ Figure A : Hasse diagram of $S'$ ]}
\end{subfigure}
\hspace{2cm}
\begin{subfigure}[b]{.4\textwidth}
\begin{tikzpicture}
   \dynkinnode{0}{0}{$S_1$}
   \dynkinarrow{0}{0.4}{0}{0.8}
   \dynkinnode{-0.5}{0.6}{$s_{\alpha_4}$}
   \dynkinnode{0}{1.2}{$S_2$}
   \dynkinarrow{0}{1.6}{0}{2}
   \dynkinnode{-0.5}{1.8}{$s_{\alpha_3}$}
   \dynkinnode{0}{2.4}{$S_3$}
   \dynkinarrow{0}{2.8}{0}{3.2}
   \dynkinnode{-0.5}{3}{$s_{\alpha_2}$}
   \dynkinnode{0}{3.6}{$S_4$}
   \dynkinarrow{-0.2}{4}{-0.8}{4.4}
   \dynkinarrow{0.2}{4}{0.8}{4.4}
   \dynkinnode{-0.8}{4}{$s_{\alpha_1}$}
   \dynkinnode{0.8}{4}{$s_{\alpha_3}$}
   \dynkinnode{-1}{4.8}{$S_5$}
   \dynkinnode{1}{4.8}{$S_6$}
   \dynkinarrow{-1}{5.2}{-1}{5.6}
   \dynkinarrow{1}{5.2}{1}{5.6}
   \dynkinarrow{0.6}{5}{-0.6}{5.8}
   \dynkinnode{-1.5}{5.4}{$s_{\alpha_3}$}
   \dynkinnode{1.5}{5.4}{$s_{\alpha_4}$}
   \dynkinnode{0}{5}{$s_{\alpha_1}$}
   \dynkinnode{-1}{6}{$S_7$}
   \dynkinnode{1}{6}{$S_8$}
   \dynkinarrow{-1}{6.4}{-1}{6.8}
   \dynkinarrow{1}{6.4}{1}{6.8}
   \dynkinarrow{-0.6}{6.2}{0.6}{7}
   \dynkinnode{-1.5}{6.6}{$s_{\alpha_2}$}
   \dynkinnode{1.5}{6.6}{$s_{\alpha_1}$}
   \dynkinnode{0}{6.2}{$s_{\alpha_4}$}
   \dynkinnode{-1}{7.2}{$S_9$}
   \dynkinnode{1}{7.2}{$S_{10}$}
   \dynkinarrow{-1}{7.6}{-1}{8}
   \dynkinarrow{1}{7.6}{1}{8}
   \dynkinarrow{-0.6}{7.4}{0.6}{8.2}
   \dynkinnode{-1.5}{7.8}{$s_{\alpha_3}$}
   \dynkinnode{1.5}{7.8}{$s_{\alpha_2}$}
   \dynkinnode{0}{7.4}{$s_{\alpha_4}$}
   \dynkinnode{-1}{8.4}{$S_{11}$}
   \dynkinnode{1}{8.4}{$S_{12}$}
   \dynkinarrow{-1}{8.8}{-1}{9.2}
   \dynkinarrow{1}{8.8}{1}{9.2}
   \dynkinnode{-1.5}{9}{$s_{\alpha_4}$}
   \dynkinnode{1.5}{9}{$s_{\alpha_3}$}
   \dynkinnode{-1}{9.6}{$S_{16}$}
   \dynkinnode{1}{9.6}{$S_{17}$}
   \dynkinarrow{-1}{10}{-1}{10.4}
   \dynkinarrow{1}{10}{1}{10.4}
   \dynkinarrow{0.6}{9.8}{-0.6}{10.6}
   \dynkinnode{-1.5}{10.2}{$s_{\alpha_3}$}
   \dynkinnode{1.5}{10.2}{$s_{\alpha_2}$}
   \dynkinnode{0}{9.8}{$s_{\alpha_4}$}
   \dynkinnode{-1}{10.8}{$S_{18}$}
   \dynkinnode{1}{10.8}{$S_{19}$}
   \dynkinarrow{-1}{11.2}{-1}{11.6}
   \dynkinarrow{1}{11.2}{1}{11.6}
   \dynkinarrow{0.6}{11}{-0.6}{11.8}
   \dynkinnode{-1.5}{11.4}{$s_{\alpha_2}$}
   \dynkinnode{1.5}{11.4}{$s_{\alpha_1}$}
   \dynkinnode{0}{11}{$s_{\alpha_4}$}
   \dynkinnode{-1}{12}{$S_{20}$}
   \dynkinnode{1}{12}{$S_{21}$}
   \dynkinarrow{-1}{12.4}{-1}{12.8}
   \dynkinarrow{1}{12.4}{1}{12.8}
   \dynkinarrow{-0.6}{12.2}{0.6}{13}
   \dynkinnode{-1.5}{12.6}{$s_{\alpha_3}$}
   \dynkinnode{1.5}{12.6}{$s_{\alpha_4}$}
   \dynkinnode{0}{12.2}{$s_{\alpha_1}$}
   \dynkinnode{-1}{13.2}{$S_{23}$}
   \dynkinnode{1}{13.2}{$S_{22}$}
   \dynkinarrow{-0.8}{13.6}{-0.2}{14}
   \dynkinarrow{0.8}{13.6}{0.2}{14}
   \dynkinnode{-0.8}{14}{$s_{\alpha_1}$}
   \dynkinnode{0.8}{14}{$s_{\alpha_3}$}
   \dynkinnode{0}{14.4}{$S_{24}$}
   \dynkinarrow{0}{14.8}{0}{15.2}
   \dynkinnode{-0.5}{15}{$s_{\alpha_2}$}
   \dynkinnode{0}{15.6}{$S_{25}$}
   \dynkinarrow{0}{16}{0}{16.4}
   \dynkinnode{-0.5}{16.2}{$s_{\alpha_3}$}
   \dynkinnode{0}{16.8}{$S_{26}$}
   \dynkinarrow{0}{17.2}{0}{17.6}
   \dynkinnode{-0.5}{17.4}{$s_{\alpha_4}$}
   \dynkinnode{0}{18}{$S_{27}$}
\end{tikzpicture}
\caption{[ Figure B : Hasse diagram of $S$ ]}
\end{subfigure}
\end{figure}

By using this relation between $\mathcal W^{P}$ and $\mathcal W^{P'}$ and the description of the Zariski tangent space $T_{x_0}S(w)$ of the Schubert variety $S(w)$ at the base point $x_0$ (Theorem 3.2 of \cite{Po}) we can show that the dimension of $T_{x_0}S(w)$ is greater than the length of $w$  unless $S(w)$ is a linear space, so that there is no smooth Schubert variety other than linear spaces in $S$. Instead of doing this, we apply the theory of the variety of minimal rational tangents again as in the previous section for the unity of the method.     \\

The semisimple part of the reductive part $L$ of $P$ is of type $B_3$ and
  the variety $\mathcal Z:=\mathcal C_{x_0}(S)$ of minimal rational tangents of $S$ at $x_0$
 is the closure of $L$-orbit of $[v_1+v_3]$ in $\mathbb P(V)$, where $V$ is the direct sum $V(\varpi_1) \oplus V(\varpi_3)$, where $V(\varpi_i)$ is the $B_3$-representation space of highest weight $\varpi_i$ for $i=1,2,3$ (see \cite{HwM05}). $\mathcal Z$ is smooth and is of Picard number one and is uniruled by lines lying on $\mathcal Z$.

 Let $z_0:=[v_1] \in \mathbb P(V)$. Then the $P$-orbit of $z_0$ is open in $\mathcal Z$ and the $L$-orbit of $z_0$ is closed. Let $Q$ denote the isotropy group  of $L$ at $z_0$.   Then the semisimple part of  the reductive part $H$  of $Q$ is of type $B_2$ and the variety $\mathcal A:=\mathcal C_{z_0}(\mathcal Z)$ of minimal rational tangents of $\mathcal Z$ at   $z_0$ is the closure of $H$-orbit of $[u_1+u_2]$, where $u_i$ is a highest weight vector of $B_2$-representation space $U(\varpi_i)$ of highest weight $\varpi_i$ for $i=1,2$.

   Let $\mathcal X$ be the closure  of a $H \cap B$-orbit in $\mathcal Z$.  As in the case of Schubert varieties, for a point $x$ in the open $H\cap B$-orbit in $\mathcal X$, we define the variety $\mathcal C_{x}(\mathcal X)$ of minimal rational tangents by the set of tangent directions of lines lying on $\mathcal X$ passing through $x$.

\begin{proposition} \label{classification F4 alpha4} Let $S=G/P$ be the rational homogeneous manifold of type $(F_4, \alpha_4)$. Then any  smooth Schubert variety of $S$ other than $S$ itself is linear.
\end{proposition}
 
\begin{proof}
Let $S_0$ be a Schubert variety of type $w$, i.e., the closure of $B$-orbit of $x_w:=wx_0$, where $w \in \mathcal W^P$.
 By   Proposition \ref{necessary conditions},  $\mathcal C_{x_w}(S_0)$    is invariant under the action of the Borel subgroup $w(L \cap B)$ of $w(L)$.
  Thus   $\mathcal Z_0:=\mathcal C_{x_0}(w^{-1}S_0)$  is invariant under the action of $L \cap B$.

Assume that $S_0$ is smooth. Then $\mathcal Z_0$ is smooth and is the closure of an $L \cap B$-orbit in $\mathcal Z$ (Proposition \ref{necessary conditions}). It suffices to show that $\mathcal Z_0$ is linear.
As in the case when $S$ is of type $(F_4, \alpha_3)$, we may be able to classify $L \cap B$-orbits in $\mathcal Z$ and to determine which closures are smooth. Instead of doing this, we will prove that the variety $\mathcal C_z(\mathcal Z_0)$ of minimal rational tangents of $\mathcal Z_0$ at a general point $z \in \mathcal Z_0$  is linear, by showing that it is the closure of a $(H \cap B)_z$-orbit in $\mathcal C_z(\mathcal Z)$ and by using that any smooth closure of $(H \cap B)$-orbit in $\mathcal A=\mathcal C_{z_0}(\mathcal Z)$ is linear.  

If $S_0$ does not contain a general line, then $\mathcal Z_0$ is contained in $ \mathcal Z \cap \mathbb P(V(\varpi_3))$ which is a rational homogeneous manifold of type $(B_3, \alpha_3)$, and thus $\mathcal Z_0$ is linear because any smooth Schubert variety of the rational homogeneous manifold of type $(B_3, \alpha_3)$ is linear.

From now on, we will assume that $S_0$   contains a general line, i.e.,  $\mathcal Z_0$ intersects $\mathcal Z^{gen}=\mathcal Z - \mathcal Z \cap \mathbb P(V(\omega_3))$ nontrivially.
$\mathcal Z_0$ is uniruled by lines in $\mathcal Z$
because $L \cap B$ has an open orbit in $\mathcal Z_0$.
Let $z=g z_0 $, where $g\in P$, be a   point in the open $L \cap B$-orbit in $\mathcal Z_0$.
By the same arguments as in the proof of Proposition 3.1 of \cite{HoM},    the variety $\mathcal A_0:=\mathcal C_{z_0}(g^{-1}\mathcal Z_0)$  of minimal rational tangents of $g^{-1}\mathcal Z_0$ at   $z_0 $ is a  smooth linear section of $\mathcal A$.  However, it is not obvious that  $\mathcal A_0$ is invariant under the action of the Borel subgroup $H \cap B$ of $H$ (the same arguments in the proof of Proposition \ref{necessary conditions} do not apply  because $\mathcal Z$ is no longer a rational homogeneous manifold).

\begin{lemma} \label{necessary conditions inductive step}
$\mathcal A_0$ is invariant under the action of $H \cap B$.
\end{lemma}

Together with the fact that $\mathcal A_0$  is smooth, we get   that $\mathcal A_0$ is  the closure of an $H\cap B$-orbit in $\mathcal A$.
Now $\mathcal A=(B_2, \alpha_1, \alpha_2)=(C_2, \alpha_2, \alpha_1)$ is the  odd symplectic Grassmannian $Gr_{\omega}(2, \mathbb C^5)$, smooth orbit closures of a Borel subgroup of $H$ (of $B_2$-type) in $\mathcal A $ other than $\mathcal A$ itself are linear. Therefore,  $\mathcal A_0$ is linear and hence $\mathcal Z_0$ is linear. Consequently, $S_0$ is linear. This completes the proof of   Proposition \ref{classification F4 alpha4}.
\end{proof}

In the remaining section we will prove Lemma  \ref{necessary conditions inductive step}. We will consider $S=G/P$ as a hyperplane section of a rational homogeneous manifold $S'=G'/P'$ associated to a long simple root, whose variety of minimal rational tangent is again a rational homogeneous manifold of Picard number one. \\

The semisimple part of the reductive part $L'$ of $P'$ is of type $D_5$
and the variety $\mathcal Z'$ of minimal rational tangents of $S'$ at $x_0$
is the $L'$-orbit of $z_0':=[v'_5]$  in $\mathbb P(V')$, where $V'$ is the $D_5$-representation space of highest weight $\varpi_5$ and $v'_5$ is  a highest weight vector in $V'$.
Since $S$ is a hyperplane section of $S'$, $\mathcal Z$ is a hyperplane section of $\mathcal Z'$, too. The reason why we introduce $S'$ is that   its variety $\mathcal Z'$ of minimal rational tangents  is   a rational homogeneous manifold, so that we can apply arguments in Section \ref{section vmrt of Schubert  varieties} to the closures of $L'\cap B'$-orbits in $\mathcal Z'$,  while the variety $\mathcal Z$ of minimal rational tangents of $S$ is not.

We will identify $z_0$ with $z_0'$ so that $\mathcal Z$ is the hyperplane section of $\mathcal Z'$ by $\mathbb P(V)$ as follows.  As a representation space of $D_4$, $V'$ is the direct sum $V''(\varpi_3) \oplus V''(\varpi_4)$, where $V''(\varpi_i)$ is the $D_4$-representation space of highest weight $\varpi_i$ for $i=1, \dots, 4$, and as a $D_4$-variety, $\mathcal Z'$ is isomorphic to the closure of $L''$-orbit of $[v''_4+v''_3]$ in $\mathbb P(V''(\varpi_4) \oplus V''(\varpi_3))$, where $v''_i$ is a highest weight vector in $V''(\varpi_i)$ for $i=1,\dots, 4$ (Proposition \ref{homogeneous horospherical varieties}). Since   $\mathcal Z'$ is homogeneous, we may identify $z_0'=[v_5']$ with $[v_4'']$. If we identify $z_0 $ with $z_0' $,
$V$ is a hyperplane of $V'$ and $\mathcal Z$ is the hyperplane section of $\mathcal Z'$ by $\mathbb P(V)$. The embedding of $\mathcal Z$ into $\mathcal Z'$ is that of $(B_3, \alpha_1, \alpha_3)$ into $(D_4, \alpha_4,\alpha_3)=(D_5, \alpha_5)$.

Let $Q'$ be the  isotropy group of $L'$ at $z_0'$.
 The semisimple part $H'$ of the reductive part of $Q'$   is of type $A_4$ and
the variety $\mathcal A'$ of minimal rational tangents of $\mathcal Z'$ at $z_0'$ is the $H'$-orbit of $[u_2']$, where $u_2'$ is a highest weight vector of $A_4$-representation space $U'$ of highest weight $\omega_2$.
 The semisimple part $H''$ of the reductive part of the isotropy group of $L''$ at $z_0'$ is of type $D_3=A_3$.
As an $A_3$-representation space $U'$ is the direct sum $U''(\varpi_1) \oplus U''(\varpi_2)$, where $U''(\varpi_i)$ is the $A_3$-representation space of highest weight $\varpi_i$ for $i=1,2,3$, and as an $A_3$-variety $\mathcal A'$ is isomorphic to the closure of $A_3$-orbit of $[u''_1+u''_2]$, where $u''_i$ is a highest weight vector of $U''(\varpi_i)$ for $i=1,2,3$ (Proposition \ref{homogeneous horospherical varieties}). As before, if we identify $[u_1]$ with $[u''_1]$, $U$ is a hyperplane of $U'$ and $\mathcal A$ is the hyperplane section of   $\mathcal A'$ by $\mathbb P(U)$.
   The embedding of $\mathcal A$ into $\mathcal A'$ is the embedding of $(B_2, \alpha_1, \alpha_2)$ into $(D_3, \alpha_3, \alpha_2)=(A_3, \alpha_1, \alpha_2)=(A_4, \alpha_2)$. \\

\noindent {\it Proof of Lemma \ref{necessary conditions inductive step}.}
Let $S_0' =S'(w')$, $w' \in \mathcal W^{P'}$,  be the  Schubert variety of $S'$ corresponding to $S_0$.
From  $S_0 = S_0'\cap \mathbb P(W) $ it follows that
 $\mathcal Z_0 =\mathcal Z_0'  \cap \mathbb P(T_{x_0}w^{-1}S_0) =\mathcal Z_0' \cap \mathbb P(T_{x_0}S)$.
 $\mathcal Z_0'$ may have more than one irreducible components, but, since $\mathcal Z_0$ is smooth,  there is an irreducible component ${\mathcal Z'_0}^0$ of $\mathcal Z_0'$ such that $\mathcal Z_0 ={\mathcal Z'_0}^0 \cap \mathbb P(T_{x_0}S)$.
   By the invariance of $\mathcal Z_0' $   under the action of $L' \cap B'$ (Proposition \ref{necessary conditions}), ${\mathcal Z'_0}^0$ is the closure of an $L'\cap B'$-orbit in $\mathcal Z'$, i.e., a Schubert variety of $\mathcal Z'$. By Proposition \ref{necessary conditions} again, for a general point $g' z_0'$ in ${\mathcal Z'_0}^0$, $\mathcal C_{g' z_o'}( {\mathcal Z'_0}^0)$ is invariant under the action of $g'(H' \cap B')$ and thus ${\mathcal A_0'}^0:=\mathcal C_{ z_o'}( {g'}^{-1}{\mathcal Z'_0}^0)$ is invariant under the action of $ H' \cap B' $. Now
 $\mathcal Z_0 
 =\mathcal Z_0'  \cap \mathbb P(T_{x_0}S)$, we have $\mathcal A_0 
 = {\mathcal A_0'}^0 \cap \mathbb P(T_{z_0} \mathcal Z ) $.   Since ${\mathcal A_0'}^0$ is invariant under the action of $H' \cap B'$, $\mathcal A_0$ is invariant under the action of $H \cap B =(H' \cap B') \cap H$.
 \qed \\

\vskip 10 pt

 \noindent {\it Proof of Theorem \ref{theorem main}.} By Proposition \ref{classification F4 alpha4}  any smooth Schubert variety of $S$ of type $(F_4, \alpha_4)$ is linear.
Now the first statement follows from  Proposition \ref{classification F4 alpha3}, and the second statement follows from Proposition \ref{rigidity F4 alpha3}.
\qed \\
 





\end{document}